\documentclass[12pt,leqno,letterpaper]{amsart}
\usepackage{amsthm}
\usepackage{mathrsfs}
\usepackage{amsmath}

\usepackage{amsmath,amsthm,amssymb,amscd}
\usepackage{latexsym}
\usepackage{hyperref}

\usepackage{enumerate}
\usepackage[utf8]{inputenc}
\usepackage[english]{babel}

\usepackage{graphicx}

\newtheorem{theorem}{Theorem}[section]
\newtheorem{proposition}{Proposition}[section]
\newtheorem{lemma}{Lemma}[section]

\newtheorem{definition}{Definition}[section]
\newtheorem{remark}{Remark}[section]

\newcommand{\tr}{{\rm Tr\hskip -0.25em}~}

\begin{document}

\title[Operator maps of Jensen-type]{Operator maps of Jensen-type}

\author[F. Hansen, H. Najafi, M.S. Moslehian]{Frank Hansen$^1$, Mohammad Sal Moslehian$^2$ \MakeLowercase{and} Hamed Najafi$^2$}

\address{\footnotesize Institute for Excellence in Higher Education, Tohoku University, Japan}
\email{frank.hansen@m.tohoku.ac.jp}
\address{Department of Pure Mathematics, Center of Excellence in Analysis on Algebraic Structures (CEAAS), Ferdowsi University of
Mashhad, P.O. Box 1159, Mashhad 91775, Iran}
\email{hamednajafi20@gmail.com} \email{moslehian@um.ac.ir
and moslehian@member.ams.org}



\maketitle

\begin{abstract}
Let $\mathbb{B}_J(\mathcal H)$ denote the set of self-adjoint operators acting on a Hilbert space $\mathcal{H}$ with spectra contained in an open interval $J$.  A map $\Phi\colon\mathbb{B}_J(\mathcal H)\to {\mathbb B}(\mathcal H)_\text{sa} $ is said to be of Jensen-type if
\[
\Phi(C^*AC+D^*BD)\le C^*\Phi(A)C+D^*\Phi(B)D
\]
for all $ A, B \in B_J(\mathcal H)$ and bounded linear operators $ C,D $ acting on $ \mathcal H $ with
$ C^*C+D^*D=I$, where $I$ denotes the identity operator. We show that a Jensen-type map on a infinite dimensional Hilbert space is of the form
$\Phi(A)=f(A)$ for some operator convex function $ f $ defined in $ J.$
\end{abstract}

\section{Introduction}

We recall that a function $f\colon J \rightarrow \mathbb{R}$ defined in a real interval $ J $ is said to be $n$-convex if the inequality
\begin{equation}\label{5}
f (\lambda A + (1-\lambda)B) \leq \lambda f (A) + (1-\lambda)f (B)
\end{equation}
holds for all $\lambda \in [0, 1]$ and operators $A, B\in{\mathbb B}_J({\mathcal H})$, when $ \dim\mathcal{H}=n. $ More generally, we call $f$ operator convex if the inequality \eqref{5} holds for all natural numbers $ n $. It is known that the inequality in this case also holds for operators on an infinite dimensional Hilbert space.
Hansen and Pedersen \cite{HP,HP2} obtained the following characterization of operator convexity.

\begin{theorem}\label{HP}
Let $f\colon J\to\mathbb{R}$ be a continuous function defined in an interval $J,$ and let $\mathcal{H}$ be an infinite dimensional Hilbert space. The following conditions are equivalent:
\begin{enumerate}[(i)]

\item f is operator convex.

\item For each natural number $k$ the inequality
\[
f\left(\sum_{i=1}^{k}C_i^{*}A_i C_i\right)\leq \sum_{i=1}^{k}C_i^{*}f(A_i)C_i
\]
holds for all $ A_1,\ldots,A_k \in B_J(\mathcal{H})$ and arbitrary operators $C_1,\ldots,C_k $ on $ \mathcal H $ with $ C_1^{*}C_1+\cdots+C_k^*C_k=I.$

\item For each natural number $k$ the inequality
\[
f\left(\sum_{i=1}^{k}P_iA_i P_i\right)\leq \sum_{i=1}^{k}P_i f(A_i)P_i
\]
\end{enumerate}
holds for all $ A_1,\ldots,A_k \in B_J(\mathcal{H})$ and projections $P_1,\ldots,P_k $ on $ \mathcal H $ with sum $ P_1+\cdots+P_k=I.$
\end{theorem}

One should note that the above result only holds on a Hilbert space of infinite dimensions. If $ \mathcal H $ is of finite dimension $ n, $ then $ f $ should be $ 2n $-convex for statement $ (ii) $ to hold with $k=2$.

There are other equivalent conditions of operator convexity, cf. \cite{AFA, 1}.
We want to study operator maps $\Phi $ given on the form $ \Phi(A)=f(A) $ in a more abstract setting, where $f(A)$ is defined by the functional calculus, and determine which general properties of $ \Phi $ that entail this particular form. A related problem is to place maps of the said form, $ \Phi(A)=f(A), $ in the context of other more general types of operator maps. To this end we first
 introduce the notion of a Jensen-type map.

\begin{definition}\label{definition of Jensen-type}
Let $ J $ be an open real interval, and let $ \mathcal H $ be a Hilbert space.
A (not necessarily linear) map $ \Phi:\mathbb{B}_J(\mathcal{H})\rightarrow \mathbb{B}(\mathcal{H})_{\text{sa}} $ is said to be of Jensen-type if
\begin{equation}\label{definition of Jensen-type map}
\Phi(C^*AC+D^*BD)\le C^*\Phi(A)C+D^*\Phi(B)D
\end{equation}
for all $ A, B \in B_J(\mathcal{H})$ and operators $ C,D $ on $ \mathcal H $  with $C^*C+D^*D=I$.
\end{definition}

Note that a Jensen-type map is convex. It is also unitarily invariant. Indeed, by choosing $ C $ as a unitary $ U $ and setting $ D=0 $, we obtain the inequality
\[
\Phi(U^*AU)\le U^*\Phi(A)U=U^*\Phi(UU^*AUU^*)U\le \Phi(U^*AU),
\]
implying that $ \Phi(U^*AU)=\Phi(U)$. We later realize that there exist unitarily invariant convex operator maps that are not of Jensen-type.

Robertson and Smith \cite{RS} showed that if $E$ is an operator system (i.e.  a closed $*$-subspace of a unital $C^*$-algebra containing the identity), $B$ is a $C^*$-algebra and a linear map $\Psi:E\otimes \mathbb{M}_n\to B\otimes\mathbb{M}_n$ satisfies $\Psi(U^*XU)=U^*\Psi(X)U$ for all $X\in E\otimes \mathbb{M}_n$ and all unitaries $U\in\mathbb{M}_n$, then there exist $\phi, \lambda: E\to B$ such that
\[
\Psi(X)=(\phi\otimes {\rm id}_n)(X) +\lambda(\tr X)\otimes I_n
\]
for all $X\in E\otimes \mathbb{M}_n$, where $I_n$ is the identity in the $C^*$-algebra $\mathbb{M}_n$ of all complex $n\times n$ matrices. In addition, Bhat \cite{BHA} proved that any bounded unitarily invariant linear map $\alpha :{\mathcal B}({\mathcal H})\to {\mathcal B}({\mathcal H})$ is of the form
\[
\alpha (X)= cX+d ~\tr X \cdot I
\]
for some $c, d\in \mathbb{C}$ if $\mathcal H$ is finite dimensional, and of the form
\[
\alpha (X)=cX
\]
for some $c\in\mathbb{C}$ if $\mathcal H$ is infinite dimensional.

\section{Unitarily invariant convex operator maps}

Let $ \Phi\colon B_J(\mathcal H)\to B(\mathcal H) $ be a unitarily invariant (not necessarily linear) map.

\begin{lemma}\label{commutation lemma}
If operators $ X\in B_J(\mathcal H) $ and $ Y\in  B(\mathcal H) $ commute, then so do $ \Phi(X) $ and $ Y. $
\end{lemma}

\begin{proof}
For any unitary operator $ U $ commuting with $ X $ we have
\[
\Phi(X)=\Phi(U^*XU)=U^*\Phi(X)U.
\]
Therefore, $ U\Phi(X)=\Phi(X)U$ and $ \Phi(X) $ is thus contained in the abelian double commutant $\{X\}''$. Since $ Y\in\{X\}' $ it follows that $ \Phi(X) $ and $ Y $ commute.
\end{proof}

In particular, if $ X=t\cdot I $ is a multiple of the identity operator for some $ t\in J, $ then we deduce that $ \Phi(t\cdot I) $ commutes with every operator in $\mathbb{B}(\mathcal{H}).$ It is therefore of the form
\begin{equation}\label{representing function}
\Phi(t\cdot I)=f(t)\cdot I, \qquad t\in J
\end{equation}
for some function $ f\colon J\to\mathbb{R}. $ We realize that $ f $ is convex if $ \Phi $ is convex.

\begin{lemma}\label{main lemma}
Let $ P_1,\dots, P_k$ be projections with $P_1+\cdots+P_k=I $ and put
\[
U=\theta P_1+\theta^2 P_2+\cdots+\theta^{k-1} P_{k-1}+P_k,
\]
where $ \theta=\exp(2\pi i/k)$ is a $k$th root of unity. Then $ U $ is unitary and
\[
\sum_{j=1}^k P_j X P_j = \frac{1}{k}\sum_{j=1}^k U^{-j} X U^j
\]
for any $ X\in \mathbb{B}(\mathcal{H}).$
\end{lemma}

\begin{proof}
Take $ r,s=1,\dots,k. $ By computation we obtain
\[
P_r\left(\sum_{j=1}^k U^{-j} X U^j\right)P_s=\sum_{j=1}^k \theta^{-jr}P_r X \theta^{js} P_s
=P_r X P_s\sum_{j=1}^k \theta^{j(s-r)} .
\]
The sum
\[
\sum_{j=1}^k \theta^{j(s-r)}=\sum_{j=1}^k \exp\Bigl(j(s-r)\frac{2\pi i}{k}\Bigr)=k
\]
for $ s=r. $ For $ s\ne r $ we set $ \omega=\exp\bigl((s-r)2\pi i/k\bigr) $ and obtain
\[
\sum_{j=1}^k \theta^{j(s-r)}=\sum_{j=1}^k \omega^j=\frac{\omega^{k+1}-\omega}{\omega-1}=0,
\]
since $ \omega\ne 1 $ and $ \omega^k=1. $ The assertion now follows.
\end{proof}

The following result is well-known for spectral functions, but it holds under the weaker conditions of only unitary invariance and convexity.

\begin{proposition}\label{main proposition}\label{main proposition}
Let $\mathcal{H}$ be a Hilbert space and $ \Phi\colon B_J(\mathcal H)\to B(\mathcal H)_\text{sa} $  a unitarily invariant convex map. Then
\[
\Phi\left(\sum_{j=1}^k P_jXP_j\right)\le\sum_{j=1}^k P_j \Phi(X) P_j
\]
for positive integers $ k $, operators $ X\in\mathbb{B}(\mathcal{H}),$ and projections $P_1, \dots, P_k  $ on $ \mathcal{H}$ with $ P_1+\dots + P_k=I. $
\end{proposition}

\begin{proof}
By repeated application of Lemma~\ref{main lemma} we obtain
\[
\begin{array}{l}
\displaystyle\Phi\left(\sum_{j=1}^k P_jXP_j\right)=\Phi\left(\frac{1}{k}\sum_{j=1}^k U^{-j} X U^j\right)
\le\frac{1}{k}\sum_{j=1}^k \Phi(U^{-j}XU^j)\\[3ex]
\displaystyle =\frac{1}{k}\sum_{j=1}^k U^{-j}\Phi(X) U^j=\sum_{j=1}^k P_j\Phi(X)P_j,
\end{array}
\]
where we used the unitary invariance and convexity of $ \Phi. $
\end{proof}

\begin{proposition}
Let $\mathcal{H}$ be a Hilbert space, and let $ \Phi\colon B_J(\mathcal H)\to B(\mathcal H)_\text{sa} $ be a unitarily invariant map. Then
\[
\Phi\left(\sum_{j=1}^k P_jXP_j\right)=\sum_{j=1}^k P_j \Phi\left(\sum_{i=1}^k P_iXP_i\right) P_j
\]
for $ X\in\mathbb{B}_J(\mathcal{H})_\text{sa}$ and projections $ P_1,\dots,P_k $ with $ P_1+\dots + P_k=I. $
 \end{proposition}
\begin{proof}
The projections $ P_1,\dots,P_k$ are necessarily mutually orthogonal and the sum
\[
\tilde X=\sum_{i=1}^k P_iXP_i
\]
commutes with $P_i $ for $ i=1,\dots,k. $ It then follows by Lemma~\ref{commutation lemma} that also
$Y=\Phi(\tilde X) $ commutes with every $P_i$ and the assertion follows.
\end{proof}

\section{The structure of Jensen-type maps}

Take an open real interval $ J, $ and let $ \Phi\colon B_J(\mathcal H)\to B(\mathcal H)_\text{sa} $ be a Jensen-type map. Since $ \Phi $ is unitarily invariant we learned in (\ref{representing function}) that $ \Phi( t\cdot I)=f(t)\cdot I, $  $ t\in J, $ for a function $ f\colon J\to\mathbb{R}. $  The convexity of $ \Phi $ implies that $ f $ is convex and thus continuous since $ J $ is open.

\begin{lemma}\label{main lemma 2}
Let  $ \Phi\colon B_J(\mathcal H)\to B(\mathcal H)_\text{sa} $ be a Jensen-type map. Then the following statements are true.

\begin{enumerate}[(i)]
\item  Let $ P $ be a projection on $ \mathcal H.$   The equality
\[
P\Phi\bigl(tP+(I-P)Y(I-P)\bigr)P=f(t)P, \qquad t\in J
\]
holds for any  $ Y\in B_J(\mathcal H). $

\item If $ \lambda $ is an eigenvalue of an operator $ X\in \mathbb{B}_J(\mathcal H) $ with corresponding eigenprojection $ P $, then
\[
P\Phi(X)P=P\Phi\bigl(\lambda P+(I-P)X(I-P)\bigr)P=f(\lambda)P.
\]

\end{enumerate}
\end{lemma}

\begin{proof}
Since $ \Phi $ is of Jensen-type we obtain
\[
\begin{array}{rl}
\Phi\bigl(tP+(I-P)Y(I-P)\bigr)&\le P\Phi(t)P+(I-P)\Phi(Y)(I-P)\\[1ex]
&=f(t)P+(I-P)\Phi(Y)(I-P).
\end{array}
\]
Furthermore,
\[
\begin{array}{rl}
f(t)&=\Phi\bigl(P\bigl(tP+(I-P)Y(I-P)\bigr)P+(I-P)t(I-P)\bigr)\\[1ex]
&\le P\Phi\bigl(tP+(I-P)Y(I-P)\bigr)P+(I-P)\Phi(t)(I-P)\\[1ex]
&\le P\bigl(f(t)P+(I-P)\Phi(Y)(I-P)\bigr)P+f(t)(I-P)\\[1ex]
&=f(t)P+f(t)(I-P)=f(t).
\end{array}
\]
We therefore have the equality
\[
f(t)= P\Phi\bigl(tP+(I-P)Y(I-P)\bigr)P +f(t)(I-P)
\]
and thus
\[
P\Phi\bigl(tP+(I-P)Y(I-P)\bigr)P=f(t)P
\]
independent of  $ Y, $ which proves $ (i). $
Statement $ (ii) $ follows from the spectral theorem and $ (i). $
\end{proof}

\begin{theorem}\label{0}
Let $ J $ be an open real interval, and let $ \mathcal H $ be a Hilbert space of finite dimension $ n. $ If $ \Phi\colon B_J(\mathcal H)\to B(\mathcal H)_\text{sa} $ is of Jensen-type, then
\[
\Phi(A)=f(A) \qquad A\in B_J(\mathcal H),
\]
where $ f $ is the function defined in (\ref{representing function}). Furthermore, $ f $ is $ n $-convex.
\end{theorem}

\begin{proof} Let $ P_1,\dots,P_k $ be the spectral projections of $ X. $
By the spectral theorem and Lemma~\ref{main lemma 2}~$(ii)$ we obtain
\[
\Phi(X)=\sum_{i=1}^kP_i\Phi(X)=\sum_{i=1}^k P_i\Phi(X)P_i=\sum_{i=1}^k f(\lambda_i)P_i=f(X),
\]
where the second equality follows from Lemma~\ref{commutation lemma}. Since $ \Phi $ is convex, it follows that $ f $ is an $ n $-convex function.
\end{proof}

Note that to obtain Theorem~\ref{0} we only used that $ \Phi $ is unitarily invariant together with the inequality in (\ref{definition of Jensen-type map}) for projections $ C=P $ and $ D=I-P. $ However, to conclude that a map of the form $ \Phi(X)=f(X) $ is of Jensen-type, we need that $ f $ is $ 2n $-convex, where $ n $ is the dimension of the underlying Hilbert space.\\[1ex]
Note also that  even if when the underlying Hilbert space is infinite dimensional the proof of the preceding theorem implies that $ \Phi(A)=f(A) $ for any finite rank operator $ A\in B_J(\mathcal H). $

\begin{lemma}\label{lemma: non-commutative convex combination}
Let $ A , Y $ be self-adjoint operators on a Hilbert space with
\[
\alpha< A\le Y
\]
for some constant $ \alpha. $
Then there exist operators $ C $ and $ D $ such that
\[
A=C^*YC + \alpha D^*D
\]
and $ C^*C+D^*D=I. $
\end{lemma}

\begin{proof}
Since $ A-\alpha>0 $ we may set  $ C=(Y-\alpha)^{-1/2}(A-\alpha)^{1/2} $ and obtain
\[
A-\alpha=C^*(Y-\alpha)C=C^*YC-\alpha C^*C.
\]
Since $ C^*C\le I $ we may put $ D=(I-C^*C)^{1/2} $ and obtain
\[
A=C^*YC+\alpha D^*D
\]
and $ C^*C+D^*D=I. $
\end{proof}

\begin{theorem}
Let $ \mathcal H $ be an infinite dimensional Hilbert space, and let $ \Phi:\mathbb{B}_J(\mathcal H)\rightarrow \mathbb{B}(\mathcal H)_{\text{sa}} $  be a Jensen-type map. Then
\[
\Phi(A)=f(A)\qquad A\in B_J(\mathcal H)\,,
\]
where $ f $ is the function defined in (\ref{representing function}). In addition, $ f $ is operator convex.
\end{theorem}

\begin{proof}
Take $ A\in \mathbb{B}_J(\mathcal H) $ and a constant $ \alpha\in J $ with $ \alpha< A. $ We may determine an upper sum operator $ Y_n=f_n(A) $ with spectrum in $ J $ by choosing  $ f_n $ as an increasing step function defined on the convex hull of the spectrum of $ A $ corresponding to a subdivision with fineness $ \varepsilon>0 $ such that $ f_n(t)= t $ in the right hand side of each subinterval. Then
\[
\alpha<A\le Y_n
\]
and $ \|Y_n-A\|\le\varepsilon $ such that $ Y_n $ converges to $ A $ in the norm topology as the fineness of the subdivision tends to zero. Furthermore, by Lemma~\ref{lemma: non-commutative convex combination}  we obtain
\[
A=C_n^*Y_nC_n+\alpha D_n^*D_n
\]
for operators $ C_n $ and $ D_n $ with $ C_n^*C_n+D_n^*D_n=I, $
and thus
\[
\Phi(A)\le C_n^*\Phi(Y_n)C_n+D_n^*\Phi(\alpha)D_n=C_n^*f(Y_n)C_n+f(\alpha)D_n^*D_n,
\]
where we first used that $ \Phi $ is of Jensen-type, and then that $ Y_n $ is a finite rank operator such that $ \Phi(Y_n)=f(Y_n). $ Notice that $ C_n $ by the spectral theorem converges to the identity operator in the norm topology, when the fineness of the subdivision tends to zero. In the limit we thus obtain $ \Phi(A)\le f(A). $ We next choose a constant $ \beta\in J $ such that
\[
\beta <Z_n\le A,
\]
where in this case $ Z_n=g_n(A) $ is an under sum operator of $ A $ with spectrum in $ J $ corresponding to a subdivision of $ J. $  We now obtain
\[
Z_n=C_n^*AC_n+\beta D_n^*D_n
\]
for operators $ C_n $ and $ D_n $ such that $ C_n^*C_n+D_n^*D_n=1, $ and $ C_n $ converges to the identity operator in the norm topology for $ n $ tending to infinity. By using that $ Y_n $ is finite rank and that $ \Phi $ is of Jensen-type we obtain the inequality
\[
f(Z_n)=\Phi(Z_n)\le C_n^*\Phi(A)C_n+\Phi(\beta)D_n^*D_n
\]
and thus in the limit $ f(A)\le\Phi(A). $
\end{proof}

Note that a Jensen-type map automatically is strongly continuous by the preceding theorem.

\begin{remark}
If $ \mathcal H $ is infinite dimensional and $ \Phi\colon B_J(\mathcal H)\to B(\mathcal H)_\text{sa} $ is unitarily invariant, we learned that the inequality
\[
\Phi(PAP+(I-P)B(I-P))\le P\Phi(A)P+(I-P)\Phi(B)(I-P)
\]
for all $ A,B\in B_J(\mathcal H) $ and projections $ P $ on $ \mathcal H$ is sufficient to conclude that $ \Phi $ is of the form $ \Phi(A)=f(A) $ for some operator convex function $ f, $ and it is therefore, by Theorem~\ref{HP},  of Jensen-type.

If $ \Phi $ is just unitarily invariant and convex, then the more restricted inequality
\[
\Phi(PAP+(I-P)A(I-P))\le P\Phi(A)P+(I-P)\Phi(A)(I-P)
\]
holds for $ A\in B_J(\mathcal H) $ and projections $ P $ on $ \mathcal H, $ cf. Proposition~\ref{main proposition}.

The difference between these two inequalities (the latter being more restricted than the former) elucidates the difference between the general class of unitarily invariant convex maps and the more restricted subset of Jensen-type maps.

 The map $ \Phi(X)=\tr X \cdot I $ is unitarily invariant and convex, but it is not of Jensen-type. To realize this, take $ A=P $ and $ B=0 $ in Definition~\ref{definition of Jensen-type}.
\end{remark}

\bibliographystyle{amsplain}

\end{document}